\documentclass[11pt]{amsart}
\usepackage{amssymb}
\theoremstyle{plain}
\newtheorem{theorem}{Theorem}
\newtheorem{corollary}{Corollary}

\newtheorem{proposition}{Proposition}

\theoremstyle{definition}

\theoremstyle{remark}

\numberwithin{equation}{section}

\setbox0=\hbox{$+$}
\newdimen\plusheight
\plusheight=\ht0
\def\+{\;\lower\plusheight\hbox{$+$}\;}

\setbox0=\hbox{$-$}
\newdimen\minusheight
\minusheight=\ht0
\def\-{\;\lower\minusheight\hbox{$-$}\;}

\setbox0=\hbox{$\cdots$}
\newdimen\cdotsheight
\cdotsheight=\plusheight
\def\cds{\lower\cdotsheight\hbox{$\cdots$}}

\begin{document}
\title[Combinatorial Identities and  $2 \times 2$ matrices ]
 {Further Combinatorial Identities deriving from the $n$-th power of a
$2 \times 2$ matrix }
\author{J. Mc Laughlin}
\address{Mathematics Department\\
 Trinity College\\
300 Summit Street, Hartford, CT 06106-3100}
\email{james.mclaughlin@trincoll.edu}
\author{ Nancy J. Wyshinski}
\address{Mathematics Department\\
       Trinity College\\
        300 Summit Street, Hartford, CT 06106-3100}
\email{nancy.wyshinski@trincoll.edu}

\keywords{}
\subjclass{}
\date{March 4th, 2004}
\begin{abstract}

In this paper we use a  formula for the $n$-th power of a
$2\times2$ matrix $A$ (in terms of the entries in $A$) to derive various combinatorial identities.
Three examples of our results follow.

1) We show that if $m$ and $n$ are positive integers and
$s \in \{0,1,2,\dots,$ $\lfloor (mn-1)/2 \rfloor \}$, then
\begin{multline*}
\sum_{i,j,k,t}2^{1+2t-mn+n}
\frac{(-1)^{nk+i(n+1)}}{1+\delta_{(m-1)/2,\,i+k}}
\binom{m-1-i}{i}
 \binom{m-1-2i}{k}\times\\
\binom{n(m-1-2(i+k))}{2j}\binom{j}{t-n(i+k)}
\binom{n-1-s+t}{s-t}\\
=\binom{mn-1-s}{s}.
\end{multline*}

2) The generalized Fibonacci
polynomial $f_{m}(x,s)$ can be expressed as
\[
f_{m}(x,s)= \sum_{k=0}^{\lfloor (m-1)/2 \rfloor}\binom{m-k-1}{k}x^{m-2k-1}s^{k}.
\]
We prove that the following functional equation holds:
\begin{equation*}
f_{mn}(x,s)=f_{m}(x,s)\times
 f_{n}\left (\,f_{m+1}(x,s)+sf_{m-1}(x,s), \,-(-s)^{m}\right) .
\end{equation*}

3) If an arithmetical function $f$ is multiplicative and for each prime $p$ there is a
complex number $g(p)$ such that
\begin{equation*}
f(p^{n+1}) = f(p)f(p^{n})- g(p)f(p^{n-1}), \hspace{15pt} n \geq 1,
\end{equation*}
then $f$ is said to be \emph{specially multiplicative}. We give another derivation of the
following formula for a specially multiplicative function $f$ evaluated at a prime power:
\begin{equation*}
f(p^{k})=\sum_{j=0}^{\lfloor k/2 \rfloor}(-1)^{j} \binom{k-j}{j}f(p)^{k-2j}g(p)^{j}.
\end{equation*}

We also prove various other combinatorial identities.
\end{abstract}

\maketitle

\section{Introduction } \label{S:intro}
Throughout the paper, let $I$ denote the $2\times2$-identity
matrix and $n$ an arbitrary positive integer.
 In \cite{McL04}, the first author proved the following theorem, which gives a formula
for the $n$-th power of a $2 \times 2$ matrix in terms of its entries:
\begin{theorem}\label{t1}
Let
\begin{equation*}
A= \left (
\begin{matrix}
a & b \\
c & d
\end{matrix}
\right )
\end{equation*}
be an arbitrary $2\times2$ matrix and let $T=a+d$ denote its trace
and $D= ad-bc$ its determinant. Let
\begin{equation*}
y_{n} = \sum_{i=0}^{\lfloor n/2 \rfloor}\binom{n-i}{i}T^{n-2
i}(-D)^{i}.
\end{equation*}
Then, for $n \geq 1$,
\begin{equation*}
A^{n}=\left (
\begin{matrix}
y_{n}-d \,y_{n-1} & b \,y_{n-1} \\
c\, y_{n-1}& y_{n}-a\, y_{n-1}
\end{matrix}
\right ).
\end{equation*}
\end{theorem}
The proof used the fact that
\begin{equation}\label{yrecur}
y_{k+1}=(a+d)y_{k}+(bc-ad)y_{k-1}.
\end{equation}
This theorem was then used to derive various binomial identities. As an
example, we cite the following corollary.

\begin{corollary}
Let $n$ be a positive integer and let $m$ be an integer with $0
\leq m \leq 2n$. Then for $ -n \leq w \leq n $,
{\allowdisplaybreaks
\begin{multline*}
\sum_{k=0}^{n-1} \binom{n-1-k}{k} \binom{n}{w+k}
\binom{k+w}{m-k-w}(-1)^{k}=\\
 \sum_{k=-2w - n + m + 1}^{m - w} \binom{n}{k+w}
\binom{n}{n+k+w-m} \binom{k+n+2w-m-1}{k}(-1)^{k} .
\end{multline*}
}
\end{corollary}
In this present paper we use Theorem \ref{t1} to derive some further identities.

\section{A Binomial Identity deriving from $(A^{m})^{n}=A^{mn}$}
We use the trivial identity $(A^{m})^{n}=A^{mn}$ to prove the following
theorem.
\begin{theorem}
Let $m$ and $n$ be positive integers and let
$s \in \{0,1,2,\dots, \lfloor (mn-1)/2 \rfloor \}$. Then
\begin{multline}\label{mns}
\sum_{i,j,k,t}2^{1+2t-mn+n}
\frac{(-1)^{nk+i(n+1)}}{1+\delta_{(m-1)/2,\,i+k}}
\binom{m-1-i}{i}
 \binom{m-1-2i}{k}\times\\
\binom{n(m-1-2(i+k))}{2j}\binom{j}{t-n(i+k)}
\binom{n-1-s+t}{s-t}=\binom{mn-1-s}{s},
\end{multline}
where $i$, $j$, $k$ and $t$ run through  integral values which
keep all binomial entries in \eqref{mns} non-negative, and
 \[
\delta_{p,\,q} =
\begin{cases}
1,&p=q,\\
0,& p \not = q.
\end{cases}
\]

\end{theorem}
\begin{proof}
Let
\[
A=\left (
\begin{matrix}
1&  1\\
x&0
\end{matrix}
\right ).
\]
From Theorem \ref{t1} and \eqref{yrecur} we have that
\begin{equation}\label{deteq}
A^{n}
=\left (
\begin{matrix}
y_{n} &  \,y_{n-1} \\
x\, y_{n-1}& y_{n}-\, y_{n-1}
\end{matrix}
\right )
=\left (
\begin{matrix}
y_{n} &  \,y_{n-1} \\
x\, y_{n-1}& x\, y_{n-2}
\end{matrix}
\right ),
\end{equation}
with
\begin{equation}\label{yeq}
y_{k}= \sum_{i=0}^{\lfloor k/2 \rfloor}\binom{k-i}{i}x^{i} = y_{k-1}+x\,y_{k-2}.
\end{equation}
Let $T_{n}$ denote the trace of $A^{n}$ and $D_{n}$ the determinant
of $A^{n}$ (so $D_{n}=(-x)^{n}$).  From \eqref{deteq}
 we have that
\begin{equation}\label{Teq}
T_{n} = y_{n}+x\,y_{n-2}.
\end{equation}
Thus the sequence $\{T_{n}\}$ satisfies the same recurrence relation as
the sequence $\{y_{n}\}$, namely
\[
T_{n+1}=T_{n}+xT_{n-1}.
\]
This leads to the explicit formula
\begin{align}\label{tneq}
T_{n}&= \left (\frac{1+\sqrt{1+4 x}}{2} \right )^{n} + \left (\frac{1-\sqrt{1+4 x}}{2} \right )^{n}\\
&=\frac{1}{2^{n+1}}
\sum_{j=0}^{\lfloor n/2 \rfloor}\sum_{k=j}^{\lfloor n/2 \rfloor}
\binom{n}{2k} \binom{k}{j}4^{j}x^{j}. \notag
\end{align}
After some straightforward but tedious calculations, we derive from the
first of these equalities, for integral $r \geq 0$, that
\begin{align}\label{tnreq}
T_{n}^{r}=
\sum_{s=0}^{\lfloor nr/2 \rfloor}\sum_{k=0}^{\lfloor r/2 \rfloor}
\sum_{i=0}^{\lfloor n(r-2k)/2 \rfloor}
\binom{r}{k} \binom{n(r-2k)}{2i}\binom{i}{s-nk}2^{1+2t-rn}
\frac{(-1)^{nk}}{1+\delta_{r/2,\,k}}x^{s}.
\end{align}
As usual,
\[
\delta_{p,\,q} =
\begin{cases}
1,&p=q,\\
0,& p \not = q.
\end{cases}
\]
For integral $j \geq 0$ define
\begin{equation}\label{jneq}
y_{j}^{(n)}= \sum_{i=0}^{\lfloor j/2 \rfloor}\binom{j-i}{i}T_{n}^{j-2i}(-D_{n})^{i}.
\end{equation}
Then Theorem \ref{t1} and the trivial identity $A^{mn} = (A^{n})^{m}$
 give that
\begin{align*}
\left (
\begin{matrix}
y_{mn} &  \,y_{mn-1} \\
x\, y_{mn-1}& xy_{mn-2}
\end{matrix}
\right )
&=
\left (
\begin{matrix}
y_{n} &  \,y_{n-1} \\
x\, y_{n-1}& xy_{n-2}
\end{matrix}
\right )^{m}\\
&=
\left (
\begin{matrix}
y_{m}^{(n)} -xy_{n-2}y_{m-1}^{(n)}&  \,y_{n-1}y_{m-1}^{(n)} \\
x\, y_{n-1}y_{m-1}^{(n)}& y_{m}^{(n)} -y_{n-1}y_{m-1}^{(n)}
\end{matrix}
\right ).
\end{align*}
If we compare $(1,2)$ entries of the first and last matrices, we have that
\begin{equation*}
y_{mn-1}=y_{n-1}y_{m-1}^{(n)}.
\end{equation*}
Upon combining \eqref{jneq},  \eqref{tnreq} and \eqref{yeq}, we get
that
\begin{multline*}
\sum_{s=0}^{\lfloor (mn-1)/2 \rfloor}\binom{mn-1-s}{s}x^{s}\\
=
\sum_{s=0}^{\lfloor (mn-1)/2 \rfloor}
\sum_{i,j,k,t}2^{1+2t-mn+n}
\frac{(-1)^{nk+i(n+1)}}{1+\delta_{(m-1)/2,\,i+k}}
\binom{m-1-i}{i}
\binom{m-1-2i}{k}\\
\times \binom{n(m-1-2(i+k))}{2j}\binom{j}{t-n(i+k)}
\binom{n-1-s+t}{s-t}
x^{s}.
\end{multline*}
Here $i$, $j$, $k$, and $t$ run through all sets of integers which keep all
binomial entries non-negative. The result now follows upon comparing
coefficients of like powers of $x$.
\end{proof}

Upon comparing like powers of $x$ on each side of \eqref{Teq}, using \eqref{yeq}
 and \eqref{tneq},
 we get the following.
\begin{corollary}
Let $n$ be a positive integer. Then for each integer $s$, $0 \leq s \leq \lfloor n/2 \rfloor$,
\begin{equation*}
\frac{1}{2^{n-2s-1}}\sum_{j=s}^{ \lfloor n/2 \rfloor} \binom{n}{2j}\binom{j}{s}
=\frac{n}{n-s} \binom{n-s}{s}.
\end{equation*}
\end{corollary}
This identity is also found in \cite{BEW99} (page 442) and \cite{G72}
(formula 3.120).

\section{A Proof of an Identity for Specially Multiplicative Functions}

An arithmetical function $f$ is said to be \emph{multiplicative} if $f(1)=1$ and
\begin{equation}\label{feq}
f(mn) = f(m)f(n),
\end{equation}
whenever $(m,n)=1$. If \eqref{feq} holds for all $m$ and $n$, then $f$ is said to
be \emph{completely multiplicative}.
A multiplicative function $f$ is said to be
\emph{specially multiplicative} if there is a completely multiplicative
function $f_{A}$ such that
\begin{equation*}
f(m)f(n) = \sum_{d|(m,n)}f \left ( \frac{mn}{d^{2}} \right ) f_{A}(d)
\end{equation*}
for all $m$ and $n$. An alternative characterization of specially
multiplicative functions is given below (see \cite{H03}, for
example):

If $f$ is multiplicative and for each prime $p$ there is a
complex number $g(p)$ such that
\begin{equation}\label{geq}
f(p^{n+1}) = f(p)f(p^{n})- g(p)f(p^{n-1}), \hspace{15pt} n \geq 1,
\end{equation}
then $f$ is specially multiplicative.
(In this case, $f_{A}(p)=g(p)$, for all primes $p$).

We give an alternative proof of the following
known result (also see \cite{H03}, for example).
\begin{proposition}
Let $f$ and $g$ be as at \eqref{geq}. Then for $k \geq 0$ and all primes $p$,
\begin{equation*}
f(p^{k})=\sum_{j=0}^{\lfloor k/2 \rfloor}(-1)^{j} \binom{k-j}{j}f(p)^{k-2j}g(p)^{j}.
\end{equation*}
\end{proposition}
\begin{proof}
Clearly we can assume $k \geq 3$.
Equation \ref{geq} implies that
{\allowdisplaybreaks
\begin{align}\label{spmuleq}
\left (
\begin{matrix}
f(p^{k})  &f(p^{k-1})   \\
 f(p^{k-1})  & f(p^{k-2})
\end{matrix}
\right )
&=
\left (
\begin{matrix}
f(p^{k-1})  &f(p^{k-2})   \\
 f(p^{k-2})  & f(p^{k-3})
\end{matrix}
\right )
\left (
\begin{matrix}
f(p)  &1   \\
-g(p)  & 0\end{matrix}
\right )\\
&=
\left (
\begin{matrix}
f(p^{2})  &f(p)   \\
 f(p)  & 1\end{matrix}
\right )
\left (
\begin{matrix}
f(p)  &1   \\
-g(p)  & 0\end{matrix}
\right )^{k-2} \notag \\
&=
\left (
\begin{matrix}
f(p)^2-g(p)  &f(p)   \\
 f(p)  & 1\end{matrix}
\right )
\left (
\begin{matrix}
f(p)  &1   \\
-g(p)  & 0\end{matrix}
\right )^{k-2}\notag \\
&=
\left (
\left (
\begin{matrix}
0 &0   \\
1+g(p)  & 0\end{matrix}
\right )+
\left (
\begin{matrix}
f(p)  &1   \\
-g(p)  & 0\end{matrix}
\right )
\right )
\left (
\begin{matrix}
f(p)  &1   \\
-g(p)  & 0\end{matrix}
\right )^{k-1}.\notag
\end{align}
}
The result now follows immediately from Theorem \ref{t1}, upon comparing
$(1,1)$ entries on each side.
\end{proof}
Remark: The Ramanujan $\tau$ function is specially multiplicative with
$g(p)$   $=$   $p^{11}$.
We note in passing that the $\tau$ Conjecture for $p$ prime,
namely that $|\tau(p)|<2p^{11/2}$, is equivalent to the conjecture that
$\lim_{k \to \infty} \tau(p^{k})/\tau(p^{k-1})$ does not exist. This follows
from \eqref{spmuleq}, the correspondence between
matrices and continued fractions and   Worpitzky's Theorem for
continued fractions.

\section{A Recurrence Formula for the Generalized Fibonacci Polynomials }

The Fibonacci polynomials $\{f_{m}(x,s)\}_{m=0}^{\infty}$ are defined by
$f_{0}(x,s)=0$, $f_{1}(x,s)=1$ and $f_{n+1}(x,s)=xf_{n}(x,s)+sf_{n-1}(x,s)$,
for $n \geq 1$. They are given explicitly by the formula
\[
f_{m}(x,s)= \sum_{k=0}^{\lfloor (m-1)/2 \rfloor}\binom{m-k-1}{k}x^{m-2k-1}s^{k}.
\]
It is clear from Theorem \ref{t1}  that the $f_{n}(x,s)$ satisfy
\begin{align*}
\left (
\begin{matrix}
 x&  1\\
s&0
\end{matrix}
\right )^{m}
&=
\left (
\begin{matrix}
f_{m+1}(x,s) &f_{m}(x,s) \\
s f_{m}(x,s)& f_{m+1}(x,s) -x f_{m}(x,s)
\end{matrix}
\right )\\
&=
\left (
\begin{matrix}
f_{m+1}(x,s) &f_{m}(x,s) \\
s f_{m}(x,s)& s f_{m-1}(x,s)
\end{matrix}
\right ).
\end{align*}
We can now use the trivial identity $A^{m\,n}=(A^{m})^{n}$ applied to the  matrix
$\left (
\begin{matrix}
 x&  1\\
s&0
\end{matrix}
\right )$, together
with Theorem \ref{t1} applied to the $(1,2)$-entries on each side to get the
following functional equation for the Fibonacci polynomials.
\begin{corollary}
Let $f_{i}(x,s)$ denote the $i$-th Fibonacci polynomial and let $m$ and $n$
be positive integers. Then
\begin{align*}
&f_{mn}(x,s)\\
&=f_{m}(x,s)\sum_{k=0}^{\lfloor \frac{n-1}{2} \rfloor}\binom{n-k-1}{k}
\big [f_{m+1}(x,s)+sf_{m-1}(x,s) \big ]^{n-2k-1}(-(-s)^{m})^{k}\\
&=f_{m}(x,s)\times
 f_{n}\left (\,f_{m+1}(x,s)+sf_{m-1}(x,s), \,-(-s)^{m}\right) .
\end{align*}
\end{corollary}

\section{A Polynomial Identity of Bhatwadekar and Roy}
In \cite{S93} Sury gave a proof of the following polynomial
identity, which he attributes to
Bhatwadekar and Roy \cite{BR91}:
\begin{corollary}
For every positive integer $n$ and all $x$,
\[
\sum_{i=0}^{\lfloor n/2 \rfloor} (-1)^i \binom{n-i}{ i} x^i (1+x)^{n-2i}
= 1 + x + \cdots + x^n.
\]
\end{corollary}
\begin{proof}
Clearly we can assume $n \geq 2$. One easily checks by induction that, for
$n \geq 2$,
{\allowdisplaybreaks
\begin{align*}
\frac{1}{1-x}
\left (
\begin{matrix}
 1-x^{n+1} &  1-x^{n} \\
 -x(1-x^{n})  & -x(1-x^{n-1})
\end{matrix}
\right )
&=
\left (
\begin{matrix}
1+x  &1   \\
-x  & 0\end{matrix}
\right )^{n}.
\end{align*}
}
The result is now immediate from Theorem \ref{t1}.
\end{proof}

\section{ Other Elementary Identities}

If we replace $n$ by $n+1$  in Equation \ref{deteq} and
take the determinant of the first and last matrices, we get
\[
(-x)^{n+1} = x( y_{n+1}y_{n-1}-y_{n}^{2}).
\]
Upon comparing coefficients of $x^{s}$, for $0 \leq s \leq n-1$ on each side, we
get the following identity.
\begin{corollary}
Let $n$ be a positive integer. If $s$ is an integer, $0 \leq s \leq n-1$, then
\begin{equation}
\sum_{j\geq 0} \binom{n-s+j}{s-j}\binom{n-j}{j}=
\sum_{j\geq 0} \binom{n+1-s+j}{s-j}\binom{n-1-j}{j}.
\end{equation}
\end{corollary}

Once again we start with the matrix $A=\left (
\begin{smallmatrix}
1&  1\\
x&0
\end{smallmatrix}
\right )$ and then consider the identity $A^{mn}=\left(A^{m}\right)^{n}=\left(A^{n}\right)^{m}$
for small values of $m$.
\begin{corollary}
Let $n$ be a positive integer and let $s$ be an integer, $0 \leq s \leq n-1$. Then
\begin{multline}\label{pwr2eq}
\sum_{i\geq 0}
\binom{n-i-1}{i}\binom{n-2i-1}{s-2i}2^{s-2i}(-1)^{i}\\
=
 \sum_{i= 0}^{\lfloor n/2 \rfloor}
\frac{n}{n-i}\binom{n-i}{i}\binom{n+i-s-1}{s-i}
=\binom{2n-s-1}{s}.
\end{multline}
\end{corollary}
\begin{proof}
With $A$ as defined above, we have
\[
A^{2}=
\left (
\begin{matrix}
x+1&  1\\
x&x
\end{matrix}
\right ).
\]
If we compare the $(1,2)$ entries of $A^{2n}$ and $(A^{2})^{n}$, using Theorem \ref{t1},
we get that
{\allowdisplaybreaks
\begin{align*}
 \sum_{s=0}^{n-1 }\binom{2n-s-1}{s}x^{s}
&=
 \sum_{i=0}^{\lfloor\frac{ n-1}{2} \rfloor}\binom{n-i-1}{i}(2x+1)^{n-2i-1}(-x^{2})^{i}\\
&=
 \sum_{i=0}^{\lfloor \frac{n-1}{2} \rfloor} \sum_{j=0}^{ n-2i-1 }
\binom{n-i-1}{i}\binom{n-2i-1}{j}2^{j}(-1)^{i}x^{2i+j}\\
&=
 \sum_{s=0}^{ n-1} \sum_{i\geq 0}
\binom{n-i-1}{i}\binom{n-2i-1}{s-2i}2^{s-2i}(-1)^{i}x^{s}.
\end{align*}
}
The equality of the first and third terms in \eqref{pwr2eq} follows on comparing powers of $x$.
On the other hand, Theorem \ref{t1} also gives that
{\allowdisplaybreaks
\begin{align*}
A^{2n}
&=
(A^{n})^{2}
=
\left (
\begin{matrix}
y_{n} & y_{n-1} \\
x\, y_{n-1}& y_{n}-y_{n-1}
\end{matrix}
\right )^{2}=
\left (
\begin{matrix}
y_{n} & y_{n-1} \\
x\, y_{n-1}& x\,y_{n-2}
\end{matrix}
\right )^{2}\\
&=\left (
\begin{matrix}
y_{n}^2+xy_{n-1}^{2} & y_{n-1}(y_{n}+x\,y_{n-2}) \\
x\, y_{n-1}(y_{n}+x\,y_{n-2})& x(y_{n-1}^2+x\,y_{n-2}^2)
\end{matrix}
\right ),
\end{align*}
}
where $y_{k}$ is as at \eqref{yeq}. It is easy to show that
\begin{equation}\label{yeqn}
y_{n}+x\,y_{n-2}=\sum_{i=0}^{\lfloor n/2 \rfloor}
\frac{n}{n-i}\binom{n-i}{i}x^{i}.
\end{equation}
If we compare the $(1,2)$ entries of $A^{2n}$ and $(A^{n})^{2}$
using \eqref{yeqn} and Theorem \ref{t1}, then
{\allowdisplaybreaks
\begin{align*}
 \sum_{s=0}^{n-1 }\binom{2n-s-1}{s}x^{s}
&=
 \sum_{k=0}^{\lfloor\frac{ n-1}{2} \rfloor}\sum_{i=0}^{\lfloor n/2 \rfloor}
\frac{n}{n-i}\binom{n-i}{i}\binom{n-k-1}{k}x^{i+k}\\
&=
 \sum_{s=0}^{ n-1} \sum_{i= 0}^{\lfloor n/2 \rfloor}
\frac{n}{n-i}\binom{n-i}{i}\binom{n+i-s-1}{s-i}x^{s}.
\end{align*}
}
The  equality of the second and third terms in \eqref{pwr2eq} now follows.
\end{proof}

A similar consideration of $A^{3n}$ and $(A^{3})^{n}$ gives the following identity.
\begin{corollary}
Let $n$ be a positive integer and $s$ an integer such that
$0 \leq s \leq \lfloor (3n-1)/2 \rfloor$. Then
\begin{equation}\label{3eq}
 \sum_{i = 0}^{\lfloor n/2 \rfloor}
3^{s-1-3i}
\binom{n-i-1}{i}\left ( \binom{n-2i}{s-3i-1} +3\binom{n-2i}{s-3i} \right)
=\binom{3n-s-1}{s}.
\end{equation}
\end{corollary}
\begin{proof}
Since
\[
A^{3}=\left (
\begin{matrix}
2x+1&  x+1\\
x^{2}+x&x
\end{matrix}
\right ),
\]
comparing the $(1,2)$ entries of $A^{3n}$ and $(A^{3})^{n}$, using Theorem \ref{t1},
gives
\begin{equation}\label{3eqa}
\sum_{s=0}^{ \lfloor\frac{3 n-1}{2} \rfloor}\binom{3n-s-1}{s}x^{s}
=(x+1)
 \sum_{i=0}^{\lfloor \frac{ n-1}{2} \rfloor }\binom{n-i-1}{i}(3x+1)^{n-2i-1}x^{3i}.
\end{equation}
The results follows, after a little simplification,
 upon comparing coefficients of like powers of $x$ on each side of \eqref{3eqa}.
\end{proof}

More generally, one can use the identity $A^{m+n}=A^{m}A^{n}$ together with
Theorem \ref{t1} to compare the $(1,1)$ entries on each side to get
(again using the notation from \eqref{yeq}) that
\begin{equation*}
y_{m+n}=y_{m}y_{n}+y_{m-1}(x\,y_{n-1}).
\end{equation*}
Upon collecting like powers of $x$ and equating coefficients on each side, we get
the following identity.

\begin{corollary}
Let $m$ and $n$ be a positive integer and $s$ an integer such that
$0 \leq s \leq \lfloor (m+n)/2  \rfloor $. Then
\begin{equation}\label{mneq}
\sum_{i \geq 0}
\binom{m-i}{i}\binom{n-s+i}{s-i}+\binom{m-i-1}{i}\binom{n-s+i}{s-i-1}
=\binom{m+n-s}{s}.
\end{equation}
\end{corollary}

\section{Concluding Remarks}
Some other interesting consequences follow readily from Theorem \ref{t1}.
We consider two more.

If we let
$A=\left (
\begin{smallmatrix}
x&  0\\
0&y
\end{smallmatrix}
\right )$, then Waring's formula
\[
x^{n}+y^{n} = \sum_{j=0}^{\lfloor n/2 \rfloor}
\frac{n}{n-j}\binom{n-j}{j}(x+y)^{n-2j}(-x\,y)^{j}
\]
can be derived easily by considering the trace of $A^{n}$.

If we set $A=\left (
\begin{smallmatrix}
x&  1\\
1&0
\end{smallmatrix}
\right )$, then Theorem \ref{t1} and the correspondence between
continued fractions and matrices give that, for $x>0$,
\[
\lim_{n \to \infty}
\frac{ \displaystyle{ \sum_{j=0}^{\lfloor n/2 \rfloor}
\binom{n-j}{j}x^{n-2j}}}
{   \displaystyle{\sum_{j=0}^{\lfloor (n-1)/2 \rfloor}
\binom{n-1-j}{j}x^{n-1-2j}}}
= \frac{2}{\sqrt{x^2+4}-x}.
\]

 \allowdisplaybreaks{

}

\begin{thebibliography}{99}

\bibitem{BEW99}
Berndt, Bruce C.;  Evans, Ronald J.;  Williams, Kenneth S.
\emph{Gauss and Jacobi sums}.
Canadian Mathematical Society Series of Monographs and Advanced Texts.
A Wiley-Interscience Publication. John Wiley \& Sons, Inc., New York, 1998. xii+583 pp.




\bibitem{BR91} S. M. Bhatwadekar, A. Roy,
\emph{Some results on embedding of a line in $3$-space.}
 J. Algebra \textbf{142} (1991), no. 1, 101--109.

\bibitem{G72}
Gould, Henry W.
\emph{Combinatorial identities.
A standardized set of tables listing 500 binomial coefficient summations.}
Henry W. Gould, Morgantown, W.Va., 1972. viii+106 pp.




\bibitem{G81} H. W. Gould,
\emph{A history of the Fibonacci $Q$-matrix and a higher-dimensional problem.}
 Fibonacci Quart. \textbf{19}  (1981), no. 3, 250--257.

\bibitem{H03}P. Haukkanen,
\emph{Some characterizations of specially multiplicative functions.}
Int. J. Math. Math. Sci. \textbf{2003}, no. 37, 2335--2344.



\bibitem{J03} R. C. Johnson,
\emph{Matrix methods for Fibonacci and related sequences.} \\
http:// maths.dur.ac.uk/~dma0rcj/PED/fib.pdf, (August 2003).

\bibitem{McL04} James Mc Laughlin,
\emph{Combinatorial Identities Deriving from the $n$-th Power of a $2 \times 2$ Matrix. }
-- submitted


\bibitem{S93} B. Sury,
\emph{A curious polynomial identity. }
Nieuw Arch. Wisk. (4) \textbf{11} (1993), no. 2, 93--96.




\bibitem{W92} Kenneth S. Williams,
\emph{The nth Power of a $2 \times 2$ Matrix (in Notes).}
 Mathematics Magazine, Vol. \textbf{65}, No. 5. (Dec., 1992), p. 336.




\end{thebibliography}
\end{document}